\documentclass{article}[11pt]
\usepackage[backend=bibtex,style=numeric,firstinits=true]{biblatex}
\usepackage{fullpage}
\usepackage{amsmath,amssymb}
\usepackage{verbatim}
\usepackage{amsthm}
\usepackage{algorithm}% http://ctan.org/pkg/algorithms
\usepackage{algpseudocode}% http://ctan.org/pkg/algorithmicx
\usepackage{graphicx}
\newcommand{\atomicNrm}[1]{\|#1\|_{\mathcal{A}}}

\usepackage{caption}
\usepackage{subcaption}
\usepackage{tikz}
\usepackage{pgfplots}
\usepackage{pgfplots}
\usepgfplotslibrary{groupplots}
\usepackage{float}
\usepackage{enumerate}

\newcommand{\bbR}{\mathbb{R}}

\numberwithin{equation}{section}
\theoremstyle{plain}%
\newtheorem{theorem}{Theorem}
\numberwithin{theorem}{section}

\newtheorem{lemma}[theorem]{Lemma}

\theoremstyle{remark}

\usepackage[colorlinks=true, urlcolor=blue, linkcolor=blue, citecolor=blue]{hyperref}

\def\bbR{\mathbb{R}}

\def\cC{{\mathcal C}}

\def\cO{{\mathcal O}}

\DeclareMathOperator*{\argmin}{arg\,min}
\DeclareMathOperator*{\minimize}{minimize}

\newcommand{\lossfn}{\ell}

\newcommand{\weight}{w}
\newcommand{\psf}{\psi}
\newcommand{\nSources}{K}
\newcommand{\reals}{\bbR}
\newcommand{\target}{y}
\newcommand{\fwd}{\Phi}
\newcommand{\dirderiv}[3]{D_{#3}#1(#2)}
\newcommand{\OurAlg}{ADCG}
\newcommand{\FCFW}{CGM-M}

\begin{document}
\title{The Alternating Descent Conditional Gradient Method \\ for Sparse Inverse Problems}
\author{Nicholas Boyd, Geoffrey Schiebinger, and Benjamin Recht}
\maketitle
\begin{abstract}{
We propose a variant of the classical conditional gradient method (CGM) for sparse inverse problems with differentiable measurement models.  Such models arise in many practical problems including superresolution, time-series modeling, and matrix completion.   Our algorithm combines nonconvex and convex optimization techniques: we propose global conditional gradient steps alternating with nonconvex local search exploiting the differentiable measurement model. This hybridization gives  the theoretical global optimality guarantees and stopping conditions of convex optimization along with the performance and modeling flexibility associated with nonconvex optimization.
Our experiments demonstrate that our technique achieves state-of-the-art results in several applications.
}
\end{abstract}

\section{Introduction}
\label{SecIntro}

A ubiquitous prior in modern statistical signal processing asserts that an observed signal is the noisy measurement of a few weighted sources.
In other words, compared to the entire dictionary of possible sources, the set of sources actually present is \emph{sparse}.
In the most abstract formulation of this prior, each source is chosen from a non-parametric dictionary, but in many cases of practical interest the sources are parameterized.
Hence, solving the sparse inverse problem amounts to finding a collection of a few parameters and weights that adequately explains the observed signal.

As a concrete example, consider the idealized task of identifying the aircraft that lead to an observed radar signal.
The sources are the aircraft themselves, and each is parameterized by, perhaps, its position and velocity relative to the radar detector.
The sparse inverse problem is to recover the number of aircraft present, along with each of their parameters.

Any collection of weighted sources can be represented as a measure on the parameter space: each source corresponds to a single point mass at its corresponding parameter value.
We will call atomic measures supported on very few points \emph{sparse} measures.
When the parameter spaces are infinite---for example the set of all velocities and positions of aircraft---the space of sparse measures over such parameters is infinite-dimensional.
This means that optimization problems searching for parsimonious explanations of the observed signal must operate over an infinite-dimensional space.

Many alternative formulations of the sparse inverse problem have been proposed to avoid the infinite-dimensional optimization required in the sparse measure setup.
The most canonical and widely applicable approach is to form a discrete grid over the parameter space and restrict the search to measures supported on the grid.
This restriction produces a finite-dimensional optimization problem~\cite{spgl1, Malioutov, JustDiscretize}.
In certain special cases, the infinite-dimensional optimization problem over measures can be reduced to a problem of moment estimation, and spectral techniques or semidefinite programming can be employed~\cite{Lieven, Prony, OffTheGrid, CandesGranda}.
More recently, in light of much of the work on compressed sensing and its generalizations, another proposal operates on atomic norms over data~\cite{CGLIP}, opening other algorithmic possibilities.

While these finite-dimensional formulations are appealing, they all essentially treat the space of sources as an unstructured set, ignoring natural structure (such as differentiability) present in many applications. 
All three of these techniques have their individual drawbacks, as well.
Gridding only works for very small parameter spaces, and introduces artifacts that often require heuristic post-processing~\cite{JustDiscretize}.  Moment methods have limited applicability, are typically computationally expensive, and, moreover, are sensitive to noise and estimates of the number of sources.
Finally, atomic norm techniques do not recover the parameters of the underlying signal, and as such are more naturally applied to denoising problems.

In this paper, we argue that all of these issues can be alleviated by returning to the original formulation of the estimation problem as an optimization problem over the space of measures.
Working with measures explicitly exposes the underlying parameter space, which allows us to consider algorithms that make local moves within parameter space.
We demonstrate that operating on the infinite-dimensional space of measures is not only feasible algorithmically, but that the resulting algorithms outperform techniques based on gridding or moments on a variety of real-world signal processing tasks.  We formalize a general approach to solving parametric sparse inverse problems via the conditional gradient method (CGM), also know as the Frank-Wolfe algorithm.  In \S\ref{SecAlgorithms}, we show how to augment the classical CGM with nonconvex local search exploiting structure in the parameter space.
This hybrid scheme, which we call the alternating descent conditional gradient method (\OurAlg), enjoys both the rapid local convergence of nonconvex programming algorithms and the stability and global convergence guarantees associated with convex optimization.
The theoretical guarantees are detailed in \S\ref{SecTheory}, where we bound the convergence rate of our algorithm and also guarantee that it can be run with bounded memory.
Moreover, in \S\ref{SecNumerics} we demonstrate that our approach achieves state-of-the-art performance on a diverse set of examples.

\subsection{Mathematical setup}
In this subsection we formalize the sparse inverse problem as an optimization problem over measures and discuss a convex heuristic.

We assume the existence of an underlying collection of objects, called sources.
Each source has a nonnegative weight $w > 0$, and a parameter $\theta \in \Theta$.
An element $\theta$ of the parameter space $\Theta$ may describe, for instance, the position, orientation, and polarization of a source. The weight may encode the intensity of a source, or the distance of a source from the measurement device.  Our goal is to recover the number of sources present, along with their individual weights and parameters.
We do not observe the sources directly, but instead are given a single, noisy measurement in $\reals^d$.

The measurement model we use is completely determined by a function $\psf: \Theta \to \reals^d$, which gives the $d$-dimensional measurement of a single, unit-weight source parameterized by a point in $\Theta$.
The measurement of a lone source is homogeneous of degree one in its weight; that is, a single source with parameter $\theta$ and weight $w > 0$ generates the measurement $w\psf(\theta) \in \reals^d$.
Finally, we assume that the measurement of a weighted collection of sources is additive.
In other words, the (noise-free) measurement of a weighted collection of sources, $\{ (\weight_i,\theta_i)\}_{i=1}^\nSources$, is simply

\begin{equation}
\label{EqnForwardOperator} \sum_{i=1}^\nSources \weight_i \psi(\theta_i) \in \reals^d.
\end{equation}
We refer to the collection $\{ (\weight_i,\theta_i)\}_{i=1}^\nSources$ as the \emph{signal parameters}, and the vector $\sum_{i=1}^\nSources \weight_i \psi(\theta_i) \in \reals^d$ as the noise-free {\em measurement}.

We can encode the signal parameters as an atomic measure $\mu$ on $\Theta$, with mass $w_i$ at point $\theta_i$:
$\mu = \sum_{i=1}^\nSources w_i \delta_{\theta_i}$.
As a consequence of the additivity and homogeneity in our measurement model, the total measurement of a collection of sources encoded in the measure $\mu$ is a linear function $\fwd$ of $\mu$:
\[
\fwd \mu  = \int \psf(\theta) d\mu(\theta).
\]
We call $\Phi$ the \emph{forward operator}.
For atomic measures of the form $\mu = \sum_{i=1}^n \weight_i \delta_{\theta_i}$, this clearly agrees with \eqref{EqnForwardOperator}; but it is
defined for all measures on $\Theta$.

\newcommand{\trueMeasure}{\mu_\mathrm{true}}
\newcommand{\supp}{\textrm{supp}}

We now introduce the sparse inverse problem as an optimization problem over measures.
Our goal is to recover $\trueMeasure$ from a measurement
$$y=\fwd \trueMeasure + \nu$$ corrupted by a noise term, $\nu$. 
Recovering the signal parameters without any prior information is, in most interesting problems, impossible; the operator $\Phi$ is almost never injective.
However, in a sparse inverse problem we have the prior belief that the number of sources present, while still unknown, is small. 
That is, we can assume that $\trueMeasure$ is an atomic measure supported on very few points.

To make the connection to compressed sensing clear, we refer to such measures as {\em sparse} measures.  Note that while we are using the language of \emph{recovery} or {\em estimation} in this section, the optimization problem we introduce is also applicable in cases where these may not be a true measure underlying the measurement model.  In~\S\ref{SecExamples} we give several examples that are not recovery problems. 

We estimate the signal parameters encoded in $\trueMeasure$ by minimizing a convex loss $\lossfn$ of the residual between $y$ and $\fwd \mu$:
\begin{equation}
\label{actualproblem}
\begin{aligned}
&\text{minimize}&& \lossfn \left(\fwd \mu - \target\right) \\
&\text{subject to}&&\mu \ge 0 \\
&&&|\textrm{supp}(\mu)| \le N, \\
\end{aligned}
\end{equation} where the optimization is over the infinite-dimensional space of measures on $\Theta$.  For example, when $\ell$ is the negative log-likelihood of the noise term $\nu$, problem \eqref{actualproblem} corresponds to maximum likelihood estimation of $\trueMeasure.$
Here $N$ is a posited upper bound on the size of the support of the true measure $\trueMeasure$, which we denote by $\textrm{supp}(\trueMeasure)$.  Although here and elsewhere in the paper we explicitly enforce the constraint that the measure be nonnegative, all of our discussion and algorithms can be easily extended to the unconstrained case.

While the objective function in \eqref{actualproblem} is convex, the constraint on the support of $\mu$ is nonconvex.  A common heuristic in these situations is to replace the nonconvex constraint with a convex surrogate. The standard surrogate for a cardinality constraint on a nonnegative measure is a constraint on the total mass. This substitution results in the standard convex approximation to~\eqref{actualproblem}:
\begin{equation}
\label{cvxproblem}
\begin{aligned}
&\minimize &&\lossfn \left(\fwd \mu - y\right)  \\
&\text{subject to} &&  \mu \ge 0 \\
&&&\mu(\Theta) \le \tau.\\
\end{aligned}
\end{equation}
\noindent Here $\tau > 0$ is a parameter that controls the total mass of $\mu$ and empirically controls the cardinality of solutions to~\eqref{cvxproblem}.
While problem~\eqref{cvxproblem} is convex, it is over an infinite-dimensional space, and it is not possible to represent an arbitrary measure in a computer.
A priori, an approximate solution to~\eqref{cvxproblem} may have arbitrarily large support, though we prove in~\S\ref{SecTheory} that we can always find solutions supported on at most $d+1$ points.
In practice, however, we are interested in approximate solutions of \eqref{cvxproblem} supported on far fewer than $d+1$ points.

A celebrated example of~\eqref{cvxproblem} occurs when $\Theta$ is the finite set $\{1, \ldots, k\}$ and $\ell(r) = \frac{1}{2}\| r \|_2^2$. In that case, a nonnegative measure over $\Theta$ can be represented as a vector $v$ in $\reals_+^k$ and the forward operator $\fwd$ as a matrix in $\reals^{d \times k}$. The total mass of the nonnegative measure $v$ is then simply $\sum_i v_i = \|v\|_1$. In this case, \eqref{cvxproblem} reduces to the nonnegative lasso.  

In this paper, we propose an algorithm for the substantially different case where $\Theta$ has some differential structure.
Our algorithm is based on a variant of the conditional gradient method that takes advantage of the differentiable nature of $\psf$, and is guaranteed to produce approximate solutions with bounded support.

\subsection{Relationship to atomic norm problems}
\label{atomicnorms}
Problems similar to~\eqref{cvxproblem} have been studied through the lens of atomic norms~\cite{CGLIP}.  
The atomic norm $\atomicNrm{\cdot}$ corresponding to a suitable collection of atoms $\mathcal{A} \subset \reals^d$
is defined as 
\[ \atomicNrm{x} = \inf \left\{ \sum_{a \in \mathcal{A}} c_a : x = \sum_{a \in \mathcal{A}}  c_a a,~c_a \ge 0 \right\}. \]
The connection to \eqref{cvxproblem} becomes clear if we take $\mathcal{A} = \{ \psi(\theta) : \theta \in \Theta\} \cup \{0\}$. 
With this choice of atomic set, we have the equality 
\[ \atomicNrm{x} = \inf \left\{ \mu(\Theta) : x = \int \psf(\theta) d\mu(\theta),~ \mu \ge 0 \right\}. \]
This equality implies the equivalence (in the sense of optimal objective value) of the infinite-dimensional optimization problem \eqref{cvxproblem} to the finite-dimensional atomic norm problem:
\begin{equation}
\label{atomicproblem}
\begin{aligned}
&\minimize &&\lossfn \left(x - y\right)  \\
&\text{subject to } && \|x\|_\mathcal{A} \le \tau. \\
\end{aligned}
\end{equation}

Much of the literature on sparse inverse problems focuses on problem \eqref{atomicproblem}, as opposed to the infinite-dimensional problem~\eqref{cvxproblem}. 
This focus is due to the fact that~\eqref{atomicproblem} has algorithmic and theoretical advantages over~\eqref{cvxproblem}.
First and foremost,~\eqref{atomicproblem} is finite-dimensional, which means that standard convex optimization algorithms may apply.
Additionally, the geometry of the atomic norm ball, $\textrm{conv} \{ \psi(\theta) : \theta \in \Theta\}$, gives clean geometric insight into when the convex heuristic will work~\cite{CGLIP}. 

With that said, we hold that the infinite-dimensional formulation we study has distinct practical advantages over the atomic norm problem~\eqref{atomicproblem}.
In many applications, it is the atomic decomposition that is of interest, 
and {\em not} the optimal point $x_\star$ of~\eqref{atomicproblem};
reconstructing the optimal $\mu_\star$ for problem~\eqref{cvxproblem} from $x_\star$ can be highly nontrivial. 
For example, when designing radiation therapy, the measure $\mu_\star$ encodes the optimal beam plan directly, while the vector $x_\star = \Phi\mu_\star$ is simply the pattern of radiation that the optimal plan produces. 
For this reason, an algorithm that simply returns the vector $x_\star$, without the underlying atomic decomposition, is not always useful in practice. 

Additionally, the measure-theoretic framework exposes the underlying parameter space, which in many applications comes with meaningful and useful structure---and is oftentimes more intuitive for practitioners than the corresponding atomic norm.
Na\"ive interpretation of the finite-dimensional optimization problem treats the parameter space as an unstructured set.
Keeping the structure of the parameter space in mind makes extensions such as \OurAlg~that make local movements in parameter space natural and uniform across applications.  

%%%%%%%%%%%%%%%%%%%%%%%%%%%%%%%%%%%%%%%%%%%
% the examples section is written in examples.tex
\section{Example applications}
\label{SecExamples}

Many practical problems can be formulated as instances of \eqref{cvxproblem}. In this section we briefly outline a few examples to motivate our study of this problem.

\paragraph{Superresolution imaging.}

The diffraction of light imposes a physical limit on the resolution of optical images. 
The goal of superresolution is to remove the blur induced by diffraction as well as the effects of pixelization and noise.
For images composed of a collection of point sources of light, this can be posed as a sparse inverse problem as follows.
The parameters $\theta_1,\ldots,\theta_\nSources$ denote the locations of $\nSources$ point sources (in $\bbR^2$ or $\bbR^3$), and $\weight_i$ denotes the intensity, or brightness, of the $i$th source.  
The image of the $i$th source is given by $w_i\psf(\theta_i)$, where $\psf$ is the pixelated point spread function of the imaging apparatus.

By solving a version of~\eqref{cvxproblem} it is sometimes possible to localize the point sources better than the diffraction limit---even with extreme pixelization.
Astronomers use this framework to deconvolve images of stars to angular resolution below the Rayleigh limit~\cite{astro}.
In biology this tool has revolutionized imaging of subcellular features~\cite{natureMethods, STORM}.
A variant of this framework allows imaging through scattering media~\cite{Liu}.
In~\S\ref{SecSMI}, we show that our algorithm improves upon the current state of the art for localizing point sources in a fluorescence microscopy challenge dataset.

\paragraph{Linear system identification.}
Linear time-invariant (LTI) dynamical systems are used to model many physical systems.
Such a model describes the evolution of an output $y_t\in \bbR$ based on the input $u_t\in\bbR$, where $t \in \mathbb{Z}_+$ indexes time.  The internal state at time $t$ of the system is parameterized by a vector $x_t\in \bbR^m$, and its relationship to the output is described by
\begin{align*}
x_{t+1} &= Ax_t + Bu_t\\
y_t &= Cx_t.
\end{align*}
Here $C$ is a fixed matrix, while $x_0, A$, and $B$ are unknown parameters.

Linear system identification is the task of learning these unknown parameters from input-output data---that is a sequence of inputs $u_1, \ldots, u_T$ and the observed sequence of outputs $y_1, \ldots, y_T$~\cite{SysID, Lieven}.
We pose this task as a sparse inverse problem.
Each source is a small LTI system with 2-dimensional state---the measurement model gives the output of the small system on the given input.
To be concrete, the parameter space $\Theta$ is given by tuples of the form $(x_0, r, \alpha, B)$ where $x_0$ and $B$ both lie in the $\ell_\infty$ unit ball in $\reals^2$, $r$ is in $[0,1]$, and $\alpha$ is in $[0,\pi]$.
The LTI system that each source describes has
$$A = r \begin{bmatrix}
    \cos(\alpha) & -\sin(\alpha) \\
 \sin(\alpha) & \cos(\alpha)
  \end{bmatrix}, \qquad C = \begin{bmatrix}
    1 & 0
   \end{bmatrix}.$$
 The mapping $\psi$ from the parameters $(x_0, r, \alpha, B)$ to the output of the corresponding LTI system on input $u_1, \ldots, u_T$ is differentiable. In terms of the overall LTI system, adding the output of two weighted sources corresponds to concatenating the corresponding parameters.

In~\S\ref{SecSMI}, we show that our algorithm matches the state of the art on two standard system identification datasets.

\paragraph{Matrix completion.}
\newcommand{\Mat}{A}
The task of matrix completion is to estimate all entries of a large matrix given observations of a few entries.
Clearly this task is impossible without prior information or assumptions about the matrix.
If we believe that a low-rank matrix will approximate the truth well, a common heuristic is to minimize the squared error subject to a nuclear norm bound.
For background in the theory and practice of matrix completion under this assumption see \cite{mc_theory_CACM}.
We solve the following optimization problem:
\[\min_{\|\Mat\|_* \le \tau} \|M(\Mat) - y \|^2. \]
Here $M$ is the masking operator, that is, the linear operator that maps a matrix $\Mat\in \bbR^{n\times m}$ to the vector containing its observed entries, and $y$ is the vector of observed entries. We can rephrase this in our notation by letting $\Theta = \{(u,v) \in \reals^n \times \reals^m : \|u\|_2 = \|v\|_2 = 1 \}$, $\psf((u,v)) = M(uv^T)$, and $\ell(\cdot) = \|\cdot \|^2.$
In~\S\ref{SecSMI}, we show that our algorithm achieves state of the art results on the Netflix Challenge, a standard benchmark in matrix completion.

\paragraph{Bayesian experimental design.}

In experimental design we seek to estimate a vector $x \in \reals^d$ from measurements of the form
\[ y_i = f(\theta_i)^Tx + \epsilon_i. \]
Here $f : \Theta \rightarrow \reals^d$ is a known differentiable feature function and $\epsilon_i$ are independent noise terms.
We want to choose $\theta_i, \ldots, \theta_k$ to minimize our uncertainty about $x$ --- if each measurement requires a costly experiment, this corresponds to getting the most information from a fixed number of experiments. For background, see~\cite{pukelsheim}.

In general, this task in intractable.
However, if we assume $\epsilon_i$ are independently distributed as standard normals and $x$ comes from a standard normal prior we can analytically derive the posterior distribution
of $x$ given $y_1, \ldots, y_k$,
as the full joint distribution of $x ,  y_1, \ldots, y_m$ is normal.

One notion of how much information $y_1, \ldots, y_m$ carry about $x$ is the entropy of the posterior distribution of $x$ given the measurements.
We can then choose $\theta_1, \ldots, \theta_k$ to minimize the entropy of the posterior, which is equivalent to minimizing the (log) volume of an uncertainty ellipsoid.
With this setup, the posterior entropy is (up to additive constants and a positive multiplicative factor) simply
\[ -\log\det \left( I + \sum_i f(\theta_i) f(\theta_i)^T \right)^{-1}.  \]

To put this in our framework, we can take $\psf(\theta) = f(\theta)f(\theta)^T$, $y = 0$ and $\ell(M) = -\log\det ( I + M)^{-1}$.
We relax the requirement to choose exactly $k$ measurement parameters and instead search for a sparse measure with bounded total mass, giving us an instance of \eqref{cvxproblem}.

\paragraph{Fitting mixture models to data.}
Given a parametric distribution $P(x | \theta)$ we consider the task of recovering the components of a mixture model from i.i.d.~samples. For background see~\cite{PGM}.
To be more precise, we are given data $\{x_1, \ldots, x_d \}$ sampled i.i.d.~from a distribution of the form
$P(x) = \int_{\theta \in \Theta}  P(x | \theta) \pi(\theta)$.
The task is to recover the mixing distribution $\pi$.
If we assume $\pi$ is sparse, we can phrase this as a sparse inverse problem.
To do so, we choose $\psf(\theta) = (P(x_i | \theta))_{i=1}^d$.
A common choice for $\ell$ is the (negative) log-likelihood of the data: i.e., $y = 0$, $\ell(p) = - \sum_i \log p_i.$
The obvious constraint is $\int d\pi(\theta) \le 1.$

\paragraph{Design of numerical quadrature rules.}
In many numerical computing applications we require fast procedures to approximate integration against a fixed measure. One way to do this is use a quadrature rule: \[ \int f(\theta) dp(\theta) \simeq \sum_{i=1}^k w_i f(x_i).\]
The quadrature rule, given by $w_i \in \reals$ and $x_i \in \Theta$, is chosen so that the above approximation holds for functions $f$ in a certain function class. The pairs $(x_i,w_i)$ are known as quadrature nodes. In practice, we want quadrature rules with very few nodes to speed evaluation of the rule.

Often we don't have an a priori description of the function class from which $f$ is chosen, but we might have a finite number of examples of functions in the class, $f_1, \ldots, f_d$, along with their integrals against $p$, $y_1, \ldots, y_d$. In other words, we know that \[\int f_i(\theta) dp(\theta) = y_i.\] A reasonable quadrature rule should  approximate the integrals of the known $f_i$ well.

We can phrase this task as a sparse inverse problem where each source is a single quadrature node. In our notation, $\psi(\theta) = (f_1(\theta), \ldots, f_d(\theta))$.  Assuming each function $f_i$ is differentiable, $\psi$ is differentiable. A common choose of $\ell$ for this application is simply the squared loss.
Note that in this application there is no need to constraint the weights to be positive.

\paragraph{Neural spike identification.}
In this example we consider the voltage $v$ recorded by an extracellular electrode implanted in the vicinity of a population of neurons.
Suppose that this population of neurons contains $K$ types of neurons,
 and that when a neuron of type $k$ fires at time $t \in \bbR$, an action potential of the form
 $\psi(t,k)$ is recorded.
 Here $\psi : \bbR \times \{1,\ldots,K\} \to \bbR^d$ is a vector of voltage samples.
 If we denote the parameters of the $i$th neuron by $\theta_i = (t_i, k_i)$, then the total voltage $v\in \bbR^d$ can be modeled as a superposition of these action potentials:
 $$v = \sum_{i=1}^\nSources w_i\psf(\theta_i).$$
 Here the weights $w_i > 0$ can encode the distance between the $i$th neuron and the electrode.
 The sparse inverse problem in this application is to recover the parameters $\theta_1,\ldots,\theta_\nSources$ and weights $w_1,\ldots,w_\nSources$ from the voltage signal $v$. For background see \cite{EkandhamTranchinaSimoncelli}.

\paragraph{Designing radiation therapy.}
External radiation therapy is a common treatment for cancer in which several beams of radiation are fired at the patient to irradiate tumors.
The collection of beam parameters (their intensities, positions, and angles) is called the treatment plan, and is chosen to minimize an objective function specified by an oncologist.
The objective usually rewards giving large doses of radiation to tumors, and low dosages to surrounding healthy tissue and vital organs.
Plans with few beams are desired as repositioning the emitter takes time---increasing the cost of the procedure and the likelihood that the patient moves enough to invalidate the plan.

A beam fired with intensity $w>0$ and parameters $\theta$ delivers a radiation dosage $w\psf(\theta)\in \bbR^d$. Here the output is interpreted as the radiation delivered to each of $d$ voxels in the body of a patient.
The radiation dosage from beams with parameters $\theta_1,\ldots,\theta_\nSources$ and intensities $w_1,\ldots,w_\nSources$ add linearly, and the objective function is convex.
For background see~\cite{radiation}.
%%%%%%%%%%%%%%%%%%%%%%%%%%%%%%%%%%%%%%%%%%%

%%%%%%%%%%%%%%%%%%%%%%%%%%%%%%%%%%%%%%%%%%%
% the algorithms sections are written in algorithms.tex
\section{Conditional gradient method}
\label{SecAlgorithms}
In this section we present our main algorithmic development.  We begin with a review of the classical conditional gradient method (CGM) for finite-dimensional convex programs.  We then translate the classical CGM for the sparse inverse problem~\eqref{cvxproblem}.  In particular, we augment this algorithm with an aggressive local search subroutine that significantly improves the practical performance of the CGM.

The classical CGM solves the following optimization problem:
\begin{equation}
\label{EqnFW}
\minimize_{x\in \cC} f(x),
\end{equation}
where $\cC$ is a closed, bounded, and convex set and $f$ is a differentiable convex function.

CGM proceeds by iteratively solving linearized versions of~\eqref{EqnFW}.  At iteration $k$, we form a linear approximation to the function $f$ at the current point $x_k$. We then minimize the linearization over the feasible set to get a potential solution $s_k$. This step can be interpreted as a restricted steepest descent: the linear functional represented by the gradient is simply the directional derivative.
As $s_k$ minimizes a simple approximation of $f$ that degrades with distance from $x_k$ we take a convex combination of $s_k$ and $x_k$ as the next iterate.  We summarize this method in Algorithm~\ref{AlgFW}.

\newcommand{\tol}{\epsilon}
\newcommand{\kmax}{k_{\tiny \max}}
\begin{minipage}{.8\textwidth}
\begin{algorithm}[H]
\caption{Conditional gradient method (CGM)}
\label{AlgFW}
\smallskip
{\bf For} $k = 1,\ldots \kmax$
\begin{enumerate}
\item Linearize: $\hat{f_k}(s)  \leftarrow f(x_{k}) + \langle \nabla f(x_{k}), s - x_{k} \rangle$. \label{linearize}
\item Minimize: $s_k \ni \argmin_{s\in \cC} \hat{f_k}(s)$.  \label{lmo}
\item Tentative update: $\tilde{x}_{k+1} \leftarrow \frac{k}{k+2}x_{k} + \frac{2}{k+2}s_k$.
\item Final update: Choose $x_{k+1}$ such that $f(x_{k+1}) \le f(\tilde{x}_{k+1})$. \label{update}
\end{enumerate}
\end{algorithm}
\end{minipage}

\bigskip
\noindent It is important to note that minimizing  $\hat{f_k}(s)$ over the feasible set $\cC$ in step~\ref{lmo} may be quite difficult and requires an application-specific subroutine.

One of the more remarkable features of the CGM is step~\ref{update}.  While the algorithm converges using the tentative update in step~\ref{update}, all of the convergence guarantees of the algorithm are preserved if one replaces $\tilde{x}_{k+1}$ with \emph{any} feasible $x_{k+1}$ that achieves a smaller value of the objective.  There are thus many possible choices for the final update in step \ref{update}, and the empirical behavior of the algorithm can be quite different for different choices.  One common modification is to do a line-search:
$$x_{k+1} = \argmin_{x \in \textrm{conv}(x_{k}, s_k)} f(x).$$
We use $\textrm{conv}$ to denote the convex hull---in this last example, a line segment.
Another variant, the {\em fully-corrective} conditional gradient method, chooses
$$x_{k+1} = \argmin_{x \in \textrm{conv}(x_{k}, s_1, \ldots, s_k)} f(x).$$
In the next section, we propose a natural choice for this step in the case of measures that uses local search to speed-up the convergence of the CGM.

One appealing aspect of the CGM is that it is very simple to compute a lower bound on the optimal value $f_\star$ as the algorithm runs. By convexity of $f$, we have
\begin{equation*}
f(s) \ge f(x_k) + \langle s - x_k, \nabla f(x_k) \rangle = \hat{f_k}(s)
\end{equation*}
for any $s \in \cC$. Minimizing both sides over $s$ gives us the elementary bound
\begin{equation*}
f_\star \ge \hat{f_k}(s_k).
\end{equation*} The right hand side of this inequality is readily computed after step \eqref{lmo}.

\subsection{CGM for sparse inverse problems}

In this section we translate the classical CGM for the sparse inverse problem~\eqref{cvxproblem}.
We give two versions---first a direct translation of the fully corrective variant and then our improved algorithm for differentiable measurement models.  To make it clear that we operate over the space of measures we change notation and denote the iterate by $\mu_k$ instead of $x_k$.  The most obvious challenge is that we cannot represent a general measure on a computer unless it is finitely-supported.  We will see however that the steps of CGM can in fact be carried out on a computer in this context.  Moreover we later prove that the iterates can be represented with bounded memory.

Before we describe the algorithm in detail, we first explain how to linearize the objective and minimize the linearization.
In the space of measures, linearization is most easily understood as a directional derivative: in the finite dimensional case, we always have that
$$\langle \nabla f(\mu_k), s \rangle = \dirderiv{f}{\mu_k}{s} := \lim_{t \downarrow 0} \frac{f(\mu_k +ts) - f(\mu)}{t}.$$
\newcommand{\rr}{r_k}

In our formulation~\eqref{cvxproblem}, $f(\mu) = \ell(\Phi\mu_k -y)$.  If we define the \emph{residual error} as $\rr = \Phi\mu_k -y$, we can compute the directional derivative of our particular choice of $f$ at $\mu_k$ as
\begin{equation}
\label{EqnDirDeriv}
\dirderiv{f}{\mu_k}{s}
= \lim_{t \downarrow 0}
\frac{\ell(\Phi( \mu_k + t s) - y) - \ell(\Phi( \mu_k) - y) }{t} = \lim_{t \downarrow 0}
\frac{\ell(\rr + t \Phi s) - \ell(\rr) }{t} = \dirderiv{\ell}{r_k}{\Phi s} =  \langle \nabla \ell(\rr) , \Phi s \rangle \,.
\end{equation}
Here, the inner product on the right hand side of the equation is the standard inner product in $\bbR^d$.

The second step of the CGM minimizes the linearized objective over the constraint set. In other words, we minimize $\langle \nabla \ell(\rr) , \Phi s \rangle$ over a candidate nonnegative measure $s$ with total mass bounded by $\tau$.
Interchanging the integral (in $\Phi$) with the inner product, and
defining $F(\theta) := \langle \nabla \ell(\rr), \psf (\theta) \rangle$, we need to solve the optimization problem:
\begin{equation}
\label{minimizeLinear}
\minimize_{s \ge 0,~s(\Theta) \le \tau} \int  F(\theta) ds(\theta).
\end{equation}
The optimal solution of~\eqref{minimizeLinear} is the point-mass $\tau \delta_{\theta_\star}$, where $\theta_\star \in \argmin F(\theta)$  (unless $F(\theta)$ is positive everywhere in which case the optimal solution is the $0$ measure).
This means that at each step of the CGM we need only add a single point to the support of our approximate solution $\mu_k$.
Moreover we prove that our algorithm produces iterates $\mu_k$ with support on at most $d+1$ points (see Theorem~\ref{ThmBoundedMemory}).

We now describe the fully-corrective variant of the CGM for sparse inverse problems (Algorithm~\ref{FW_Meas}).
The state of the algorithm at iteration $k$ is a nonnegative atomic measure
$\mu_k$ supported on a finite set $S_k$ with mass $\mu_k(\{\theta\})$ on points $\theta \in S_k$.  The algorithm alternates between selecting a source to add to the support, and tuning the weights to lower the current cost.
This tuning step (Step \ref{computeW}) is a finite-dimensional convex optimization problem that we can solve with an off-the-shelf algorithm.

\newcommand{\FWM}{\FCFW}  
\newcommand{\outY}{{\hat y}}
\newcommand\itemrow[2]{%
  \item\makebox[14em][l]{#1}%
    \makebox[10em][l]{#2}%
}
\begin{minipage}{.8\textwidth}
\begin{algorithm}[H]
\caption{Conditional gradient method for measures (\FWM)}
\label{FW_Meas}
{\bf For} $k = 1 : \kmax$
\begin{enumerate}
\itemrow{Compute gradient of loss:}{$g_k = \nabla \ell(\Phi \mu_{k-1} - y).$}
\itemrow{Compute next source:}{$\theta_k \in \argmin\limits_{\theta \in \Theta}   \langle g_k, \psf(\theta) \rangle.$} \label{gba}
\itemrow{Update support:}{$S_{k} \leftarrow S_{k-1} \cup \{ \theta_k \}.$}
\itemrow{Compute weights:}{$\mu_{k} \leftarrow \argmin\limits_{\substack{\mu \ge 0, \; \mu(S_k) \le \tau \\ \mu(S_k^c) = 0}}   \ell \left(\sum_{\theta \in S_k} \mu(\{\theta\}) \psf(\theta) - y\right).$ \label{computeW}}
\itemrow{Prune support:}{$S_{k} \gets \supp (\mu_{k}).$}
\end{enumerate}
\end{algorithm}
\end{minipage}

\bigskip

While we can simply run for a fixed number of iterations, we may stop early using the standard CGM bound.
With a tolerance parameter $\epsilon>0$,
we terminate when the conditional gradient bound assures us that we are at most $\epsilon$-suboptimal. In particular, we terminate when
\begin{equation}
\label{bound}
\langle \Phi \mu_k, g_k \rangle  - \tau \langle \psf(\theta_k), g_k \rangle_+ < \epsilon.
\end{equation}

Unfortunately, \FWM~does not perform well in practice. Not only does it converge very slowly, but the solution it finds is often supported on an undesirably large set.
As illustrated in Figure~\ref{FigFullyCorrective}, the performance of \FWM~is limited by the fact that it can only change the support of the measure by adding and removing points; it cannot smoothly move $S_k$ within $\Theta$.   Figure~\ref{FigFullyCorrective} shows \FWM~applied to an image of two closely separated sources. The first source $\theta_1$ is placed in a central position overlapping both true sources.  In subsequent iterations sources are placed too far to the right and left, away from the true sources.
To move the support of the candidate measure requires \FWM{} to repeatedly add and remove sources; it is clear that the ability to move the support smoothly within the parameter space would resolve this issue immediately.

\begin{figure}
    \centering
    \begin{subfigure}[b]{0.3\textwidth}
        \centering
        \includegraphics[width=\textwidth]{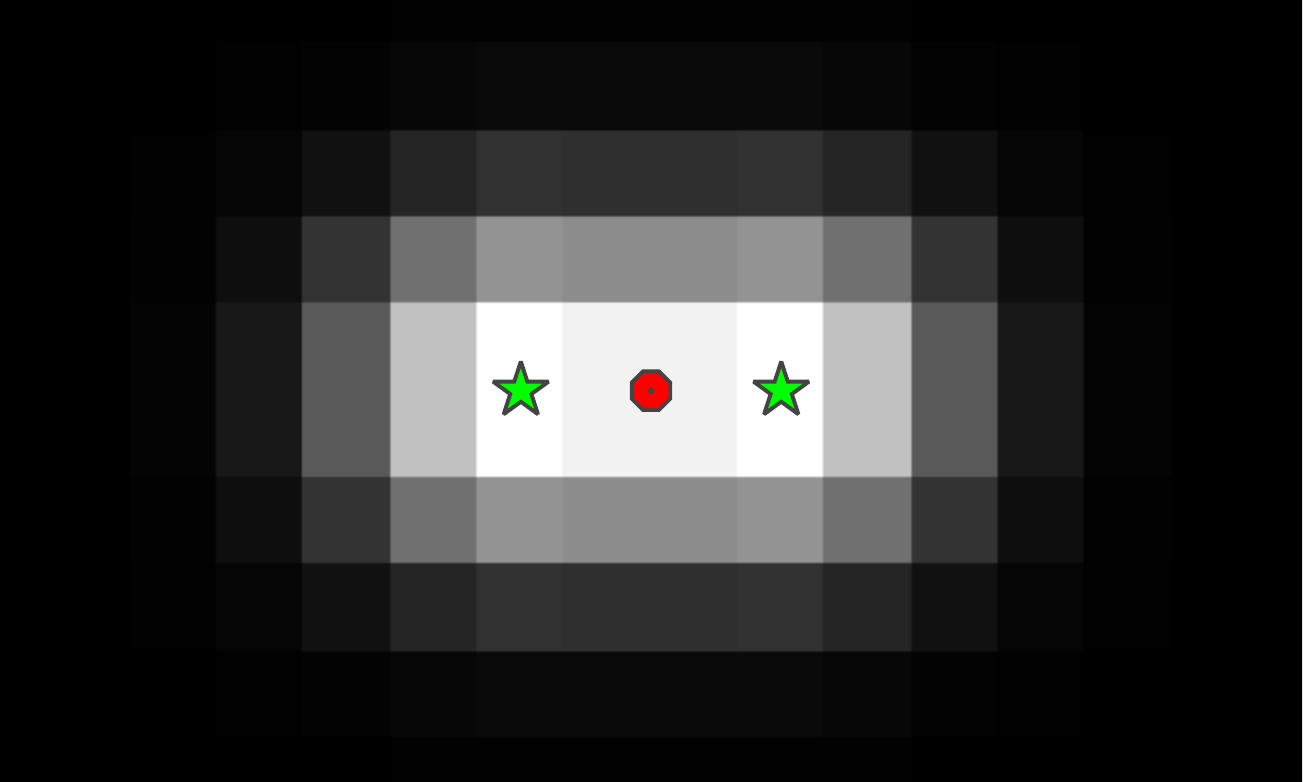}
        \label{figb}
    \end{subfigure}
    \begin{subfigure}[b]{0.3\textwidth}
        \centering
        \includegraphics[width=\textwidth]{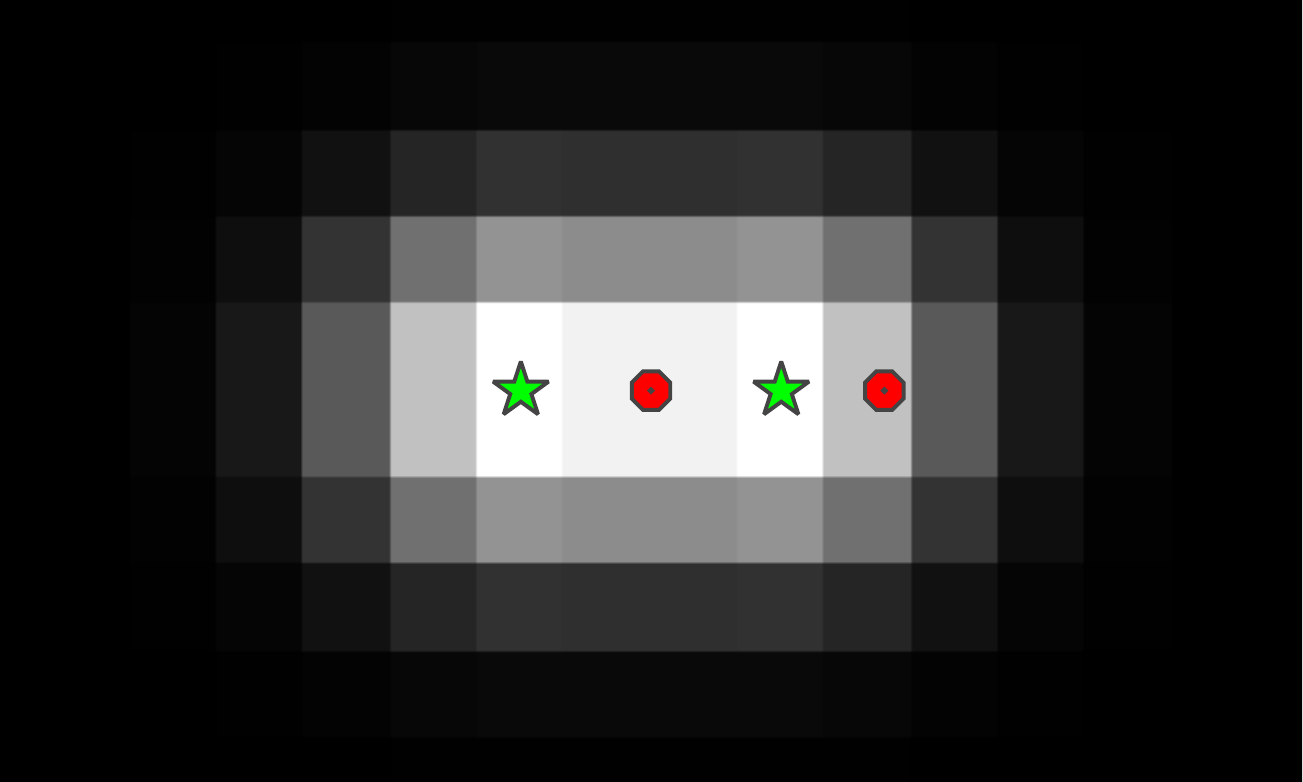}
        \label{figc}
    \end{subfigure}
    \begin{subfigure}[b]{0.3\textwidth}
        \centering
        \includegraphics[width=\textwidth]{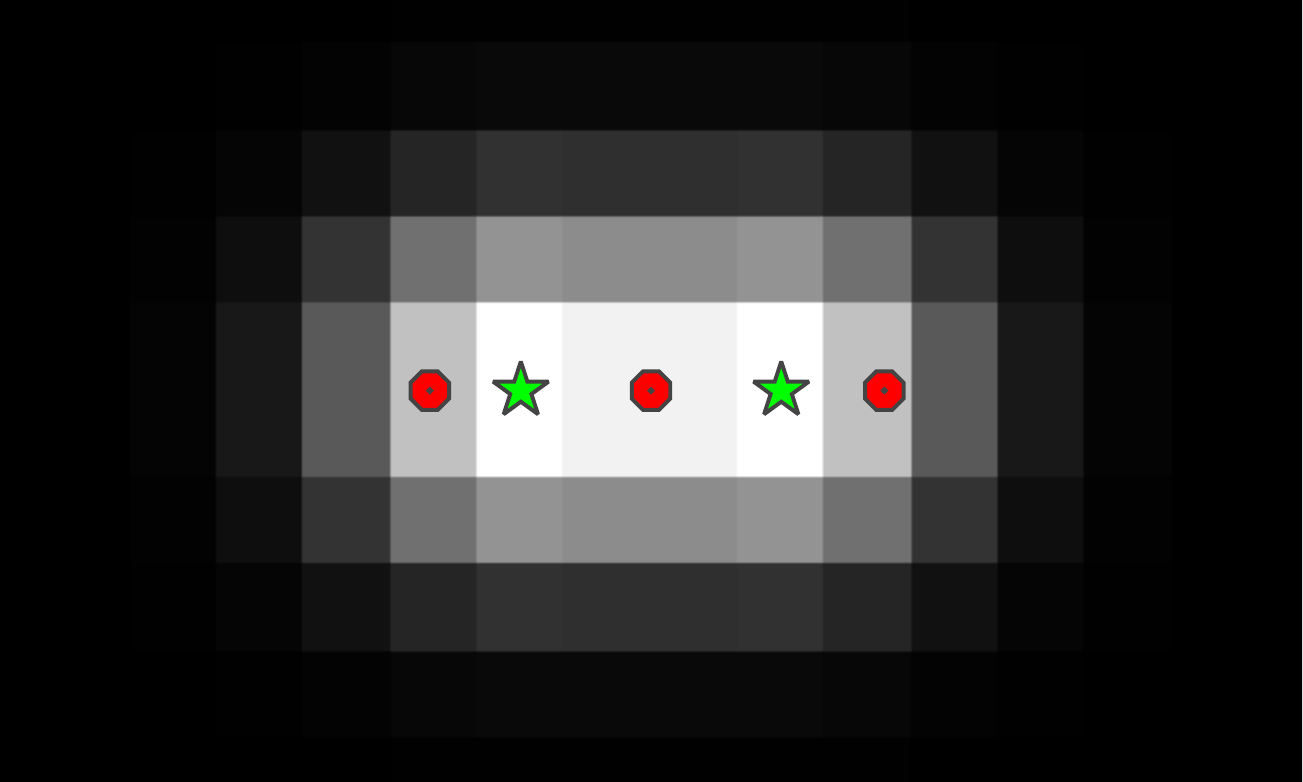}
        \label{figd}
    \end{subfigure}
    \caption{The three plots above show the first three iterates of the fully corrective CGM in a simulated superresolution imaging problem with two point sources of light.  The locations of the true point sources are indicated by green stars, and the greyscale background shows the pixelated image.  The elements of $S_k$ for $k = 1,2,3$ are displayed by red dots.}
    \label{FigFullyCorrective}
\end{figure}

In practice, we can speed up convergence and find significantly sparser solutions by allowing the support to move continuously within $\Theta$.   The following algorithm, which we call the alternating descent conditional gradient method (\OurAlg), exploits the differentiability of $\psf$ to locally improve the support at each iteration.

\newcommand{\improve}{\textbf{local\_descent}}
\begin{minipage}{.8\textwidth}
\begin{algorithm}[H]
\caption{Alternating descent conditional gradient method (\OurAlg)}
\label{FW_ours}
\smallskip
{\bf For} $k = 1:\kmax$
\begin{enumerate}
\itemrow{Compute gradient of loss: }{$g_k \gets \nabla \ell(\Phi \mu_{k-1} - y).$}
\itemrow{\label{gba}Compute next source: }{Choose $\theta_k \in \argmin\limits_{\theta \in \Theta}  \langle \psf(\theta), g_k \rangle.$}
\itemrow{Update support: }{$S_{k} \gets S_{k-1} \cup \{ \theta_k \}$.}
\item \label{update}  Coordinate descent on nonconvex objective:\\
{\bf Repeat:}
    \begin{enumerate}
    \itemrow{Compute weights: }{$\mu_{k} \leftarrow \argmin\limits_{\substack{\mu \ge 0, \; \mu(S_k) \le \tau \\ \mu(S_k^c) = 0}}   \ell \left(\sum_{\theta \in S_k} \mu(\{\theta\}) \psf(\theta) - y\right).$}
    \itemrow{Prune support: }{$S_{k} = \text{support}(\mu_k)$.}
\itemrow{Locally improve support: }{$S_{k} = \improve((\theta, \mu_k(\{\theta \})) : \theta \in S_k )$.}
    \end{enumerate}
\end{enumerate}
\end{algorithm}
\end{minipage}

\bigskip

\noindent Here $\improve$ is a subroutine that takes a measure $\mu_k$ with atomic representation $(\theta_1,w_1),\ldots,(\theta_m,w_m)$ and attempts to use gradient information to reduce the function 
$$(\theta_1,\ldots,\theta_m) \mapsto \ell (\sum_{i=1}^m w_i \psf(\theta_i) - y),$$
holding the weights fixed.

When the number of sources is held fixed, the optimization problem
\begin{equation}
\label{nonconvex}
\begin{aligned}
&\text{minimize }&& \lossfn \left(\sum_{i=1}^k w_i \psi(\theta_i) - \target\right) \\
&\text{subject to }&& w_i \ge 0 \\
&&&\theta_i \in \Theta \\
&&&\sum_i w_i \le \tau \\
\end{aligned}
\end{equation}
is nonconvex.
Step~\ref{update} is then block coordinate descent over $w_i$ and $\theta_i$.
The algorithm as a whole can be interpreted as alternating between performing descent on the convex (but infinite-dimensional) problem~\eqref{cvxproblem} in step~\ref{gba} and descent over the finite-dimensional (but nonconvex) problem~\eqref{nonconvex} in step~\ref{update}.  The bound~\eqref{bound} remains valid and can be used as a termination condition.

As we have previously discussed, this nonconvex local search does not change the convergence guarantees of the CGM whatsoever.
We will show in Section~\ref{SecTheory} that this is an immediate consequence of the existing theory on the CGM.
However, as we will show in ~\S\ref{SecNumerics}, the inclusion of this local search dramatically improves the performance of the CGM.

\subsection{Interface and implementation}
Roughly speaking, running \OurAlg~on a concrete instance of~\eqref{cvxproblem} requires subroutines for two operations.
We need algorithms to approximately compute:
\begin{enumerate}[(a)]
\item $\psi(\theta)$ and $\frac{d}{d\theta} \psi(\theta)$ \quad for $\theta \in \Theta$. \label{psf}
\item  $\argmin\limits_{\theta \in \Theta}  \langle \psf(\theta), v \rangle$ for arbitrary vectors $v \in \reals^d$.   \label{LMO}
\end{enumerate}
Computing \eqref{psf} is usually straightforward in applications with differentiable measurement models.
Computing \eqref{LMO} is not easy in general. However, there are many applications of interest where \eqref{LMO} is tractable.  For example, if the parameter space  $\Theta$ is low-dimensional, then  the ability to compute~\eqref{psf} is sufficient to approximately compute~\eqref{LMO}: we can simply grid the parameter space and begin local search using the gradient of the function $\theta \mapsto \langle \psi(\theta), v \rangle$.
Note that because of the local improvement step, \OurAlg~works well even without exact minimization of~\eqref{LMO}.  We prove this fact about inexact minimization in Section~\ref{SecTheory}.

If the parameter space is higher dimensional, however, the feasibility of computing~\eqref{LMO} will depend on the specific application.  One example of particular interest that has been studied in the context of the CGM is matrix completion~\cite{jaggimc,cogent,harchaoui2014conditional,AcceleratedMatrix}.  In this case, the~\eqref{LMO} step reduces to computing the leading singular vectors of a sparse matrix.  We will show that adding local improvement to the CGM accelerates its convergence on matrix completion in the experiments.

We also note that in the special case of linear system identification, $\Theta$ is 6 dimensional, which is just large enough such that gridding is not feasible.  In this case, we show that we can reduce the 6-dimensional optimization problem to a 2-dimensional problem and then again resort to gridding.  We expect that in many cases of interest, such specialized solvers can be applied to solve the selection problem~\eqref{LMO}.

%%%%%%%%%%%%%%%%%%%%%%%%%%%%%%%%%%%%%%%%%%%

\section{Related work}

There has recently been a renewed interest in the conditional gradient method as a general purpose solver for constrained inverse problems~\cite{jaggi,harchaoui2014conditional}.  These methods are simpler to implement than the projected or proximal gradient methods which require solving a quadratic rather than linear optimization over the constraint set.

The idea of augmenting the classic conditional gradient method with improvement steps is not unique to our work.
Indeed, it is well known that any modification of the iterate that decreases the objective function will not hurt theoretical convergence rates~\cite{jaggi}.
Moreover, Rao \emph{et al}~\cite{cogent} have proposed a version of the conditional gradient method, called CoGENT, for atomic norm problems that take advantage of many common structures that arise in inverse problems.
The reduction described in our theoretical analysis makes it clear that our algorithm can be seen as an instance of CoGENT specialized to the case of measures and differentiable measurement models.

The most similar proposals to~\OurAlg~come from the special case of matrix completion or nuclear-norm regularized problems. Several papers \cite{AcceleratedMatrix, Hazan, harchaoui2014conditional, jaggimc} have proposed algorithms based on combinations of rank-one updates and local nonconvex optimization inspired by the well-known heuristic of~\cite{BM}.
While our proposal is significantly more general, \OurAlg~essentially recovers these algorithms in the special case of nuclear-norm problems.

We note that in the context of inverse problems, there are a variety of algorithms proposed to solve the general infinite-dimensional problem~\eqref{cvxproblem}.
Tang \emph{et al}~\cite{JustDiscretize} prove that this problem can be approximately solved by gridding the parameter space and solving the resulting finite dimensional problem.  
However, these gridding approaches are not tractable for problems with parameter spaces even of relatively modest dimension.
Moreover, even when gridding is tractable, the solutions obtained are often supported on very large sets and heuristic post-processing is required to achieve reasonable performance in practice~\cite{JustDiscretize}.
In spite of these limitations, gridding is the state of the art in many application areas including computational neuroscience~\cite{EkandhamTranchinaSimoncelli},
superresolution fluorescence microscopy~\cite{FastStorm},
radar~\cite{BaraniukSteeghs,HermanStrohmer},
remote sensing~\cite{FannjiangStrohmer},
compressive sensing~\cite{BajwaHaupt, MalioutovCetin,DuarteBaraniuk},
and polynomial interpolation~\cite{Rauhut}.

There have also been a handful of papers that attempt to tackle the infinite-dimensional problem without gridding.
For the special case where $\ell (\cdot) = \| \cdot \|_2^2$, Bredies and Pikkarainen~\cite{Austrian} propose an algorithm to solve the Tikhonov-regularized version of problem~\eqref{cvxproblem} that is very similar to Algorithm~\ref{FW_ours}.
They propose performing a conditional gradient step to update the support of the measure, followed by soft-thresholding to update the weights.
Finally, with the weights of the measure fixed they perform discretized gradient flow over the locations of the point-masses.
However, they do not solve the finite-dimensional convex problem at every iteration, which means there is no guarantee that their algorithm has bounded memory requirements.
For the same reason, they are limited to one pass of gradient descent in the nonconvex phase of the algorithm.
In~\S\ref{SecNumerics} we show that this limitation has serious performance implications in practice.

%%%%%%%%%%%%%%%%%%%%%%%%%%%%%%%%%%%%%%%%%%%
% the theory section is written in theory.tex

\section{Theoretical guarantees}
\label{SecTheory}
In this section we present a few theoretical results.
The first guarantees that we can run our algorithm with bounded memory.
The second result guarantees that the algorithm converges to an optimal point and bounds the worst-case rate of convergence.

\subsection{Bounded memory}
As the CGM for measures adds one point to the support of the iterate per iteration, we know that the cardinality of the support of $\mu_k$ is bounded by $k$.
For large $k$, then, $\mu_k$ could have large support.
The following theorem guarantees that we can run our algorithm with bounded memory and in fact we need only store at most $d+1$ points, where $d$ is the dimension of the measurements.
\begin{theorem}
\label{ThmBoundedMemory}
\OurAlg~may be implemented to generate iterates with cardinality of support uniformly bounded by $d+1$.
\end{theorem}

\begin{proof}
Lemma~\eqref{lemma1} allows us to conclude that the fully-corrective step ensures that the support of the measure remains bounded by $d+1$ for all iterations.
\end{proof}

\newcommand{\wopt}{w_\star}
\newcommand{\uopt}{u_\star}
\begin{lemma}
\label{lemma1}
The finite-dimensional problem
\begin{equation}
\label{fdprob}
\minimize_{w \ge 0,~\sum_{i=1}^{m} w_i \le \tau} \ell(\sum_i w_i \psf(\theta_i) - y)
\end{equation}
has an optimal solution $\wopt$ with at most $d+1$ nonzeros.
\end{lemma}
\begin{proof}

Let ${\uopt}$ be any optimal solution to~\eqref{fdprob}.
As ${\uopt}$ is feasible, we have that
\[v = \sum_i {\uopt}_i \psf(\theta_i) \in \tau \textrm{conv}(\{\psf(\theta_i) : i = 1,\ldots,m\} \cup\{0\}).\]
In other words, $\frac{v}{\tau}$ lies in the convex hull of a set in $\reals^d$.
Caratheodory's theorem immediately tells us that $\frac{v}{\tau}$ can be represented as a convex combination of at most $d+1$ points from $\{\psf(\theta_i) : i = 1,\ldots,m\}$.
That is, there exists a $\wopt$ with at most $d+1$ nonzeros such that
\[ \sum_{i=1}^{m} {\wopt}_i \psf(\theta_i) = v.\]
This implies that $\wopt$ is also optimal for~\eqref{fdprob}.
\end{proof}
Note that in order to find $\wopt$, we need to either use a simplex-type algorithm to solve \eqref{fdprob} or explore the optimal set using the random ray-shooting procedure as described in~\cite{subopt}.

\subsection{Convergence analysis}
\newcommand{\diam}{B}
\newcommand{\Lgrad}{L}
\newcommand{\curv}{C_{f,\mathcal{S}}}
\newcommand{\conv}{\textrm{conv}}
\newcommand{\domain}{\conv\mathcal{A}}
\newcommand{\linac}{\zeta}
\newcommand{\diamTheta}{B_\theta}
We now analyze the worst-case convergence rate for~\OurAlg~applied to~\eqref{cvxproblem}.
Theorem~\ref{ThmConv} below guarantees that~\OurAlg~achieves accuracy $\delta$ in $\cO ( \frac 1 \delta ) $ iterations.

The theorem applies even when the linear minimization step is performed approximately.
That is, we allow $\theta_k$ to be chosen such that
\begin{equation}
\label{approx}
 \langle \psf(\theta_k) , g_k \rangle \le \min_{\theta \in \Theta} \langle \psf(\theta), g_k \rangle  + \frac{\linac}{k+2}
 \end{equation}
for some $\linac \ge 0$.
When inequality~\eqref{approx} holds, we say that the linear minimization problem in iteration $k$ is solved to precision $\linac$.

The analysis relies on a finite-dimensional optimization problem equivalent to~\eqref{cvxproblem}. Let $\mathcal{A} = \{ \psi(\theta) : \theta \in \Theta\} \cup \{0\}$.
Readers familiar with the literature on atomic norms~\cite{CGLIP} will recognize the finite-dimensional problem we consider as an atomic norm problem:
\begin{equation}
\label{finiteproblem}
\begin{aligned}
&\minimize_{x \in \reals^d} &&\lossfn (x - y)  \\
&\text{subject to } &&  {x \in \tau \domain}.
\end{aligned}
\end{equation}

The connection to~\eqref{cvxproblem} becomes clear if we note that $\tau \domain = \{ \Phi \mu : \mu \ge 0, \mu(\theta) \le \tau\}$. Any feasible measure $\mu$ for~\eqref{cvxproblem} gives us a feasible point $\Phi \mu$ for~\eqref{finiteproblem}. Likewise, any feasible $x$ for~\eqref{finiteproblem} can be decomposed as a feasible measure $\mu$ for~\eqref{cvxproblem}. Furthermore, these equivalences preserve the objective value.

Before we state the theorem precisely, we introduce some notation.
Let $\ell_\star = \ell(\fwd \mu_\star - y)$ denote the optimal value of~\eqref{cvxproblem}---the discussion above implies that $\ell_\star$ is also the optimal value of~\eqref{finiteproblem}.
Following Jaggi in~\cite{jaggi}, we define the curvature parameter $\curv$ of a function $f$ on a set $\mathcal{S}$.
Intuitively, $\curv$ measures the maximum divergence between $f$ and its first-order approximations, $\hat{f}(z; x) =  f(x) + \langle z - x, \nabla f(x) \rangle$:
$$\curv = \sup_{\substack{x,s \in \mathcal{S} \\ \gamma \in [0,1]\\ z= x+\gamma (s-x)}} \frac 2 {\gamma^2} (f(z) - \hat{f}(z; x)).$$

\begin{theorem}
\label{ThmConv}
Let $C$ be the curvature parameter of the function $f(x)=\ell(x - y)$ on the set $\tau \domain$.
If each linear minimization subproblem is solved to precision $C\linac$,
the iterates $\mu_1, \mu_2, \ldots$ of \OurAlg~applied to~\eqref{cvxproblem} satisfy
\begin{equation*}
\label{EqnThm}
\ell(\fwd \mu_k -y) - \ell_\star  \le \frac{2C}{k+2}(1+\linac).
\end{equation*}
\end{theorem}

\begin{proof}

We first show that the points $\Phi \mu_1, \Phi \mu_2, \ldots$ are iterates of the standard CGM (with a particular choice of the final update step) applied to the finite-dimensional problem~\eqref{finiteproblem}.
We then appeal to~\cite{jaggi} to complete the proof.

Suppose that $\Phi \mu_k = x_k$. We show that the linearization step in both algorithms produces the same result (up to the equivalence mentioned earlier).
Let $\theta_{k+1} = \argmin_{\theta \in \Theta} \langle \psf(\theta), \nabla \ell(\Phi \mu_k - y) \rangle$ be the output of step~\ref{gba} of \OurAlg.  Let $s_k$ be the output of the linear minimization step of the standard CGM applied to~\eqref{finiteproblem} starting at $x_k$. Then
\[s_k = \argmin_{s \in \tau \domain} \langle s, \nabla \ell(x_k - y) \rangle.\]
Recalling that $\domain =\{ \Phi \mu \mid \mu \ge 0,~\mu(\Theta) \le 1\}$, we must have $s_k = \tau \psf (\theta_k)$.  Therefore, the linear minimization steps of the standard CGM and \OurAlg~coincide.

We now need to show that the nonconvex coordinate descent step in \OurAlg~ is a valid final update step for the standard CGM applied to~\eqref{finiteproblem}. This is clear as the coordinate descent step does at least as well as the fully-corrective step.  We can hence appeal to the results of Jaggi~\cite{jaggi} that bound the convergence rate of the standard CGM on finite-dimensional problems to finish the proof.
\end{proof}

%%%%%%%%%%%%%%%%%%%%%%%%%%%%%%%%%%%%%%%%%%%

%%%%%%%%%%%%%%%%%%%%%%%%%%%%%%%%%%%%%%%%%%%
% the numerics section is written in numerics.tex

\section{Numerical results}
\label{SecNumerics}

In this section we apply \OurAlg~to three of the examples in~\S\ref{SecExamples}:
superresolution fluorescence microscopy, matrix completion, and system identification. 
We have made a simple implementation of \OurAlg~publicly available on github: 
\begin{center}
\url{https://github.com/nboyd/SparseInverseProblems.jl}.
\end{center} 
This allows the interested reader to follow along with these examples, and, hopefully, to apply \OurAlg~to other instances of~\eqref{cvxproblem}.

For each example we briefly describe how we implement the required subroutines for \OurAlg, though again the interested reader may want to consult our code for the full picture.
We then describe how \OurAlg~compares to prior art.
Finally, we show how \OurAlg~improves on the standard fully-corrective conditional gradient method for measures (\FWM{}) and a variant of the gradient flow algorithm (GF) proposed in~\cite{Austrian}.
While the gradient flow algorithm proposed in~\cite{Austrian} does not solve the finite-dimensional convex problem at each step, our version of GF does.
We feel that this is a fair comparison: intuitively, fully solving the convex problem can only improve the performance of the GF algorithm.
All three experiments require a subroutine to solve the finite-dimensional convex optimization problem over the weights.
For this we use a simple implementation of a primal-dual interior point method, which we include in our code package.

For each experiment we select the parameter $\tau$ by inspection. For matrix completion and linear system ID this means using a validation set. For single molecule imaging each image requires a different value of $\tau$. For this problem, we run \OurAlg~ with a large value of $\tau$ and stop when the decrease in the objective function gained by the addition of a source falls below a threshold. This heuristic can be viewed as post-hoc selection of $\tau$ and the stopping tolerance $\epsilon$, or as a stagewise algorithm~\cite{tibshirani}.

The experiments are run on a standard c4.8xlarge EC2 instance.
Our naive implementations are meant to demonstrate that \OurAlg~is easy to implement in practice and finds high-quality solutions to~\eqref{cvxproblem}.
For this reason we do not include detailed timing information.

\subsection{Superresolution fluorescence microscopy}
\label{SecSMI}
We analyze data from the Single Molecule Localization Microscopy (SMLM) challenge~\cite{SMIpaper,SMIChallenge}.  
Fluorescence microscopy is an imaging technique used in the biological sciences to study subcellular structures in vivo.  
The task is to recover the 2D positions of a collection of fluorescent proteins from images taken through an optical microscope.  

Here we compare the performance of our \OurAlg~to the gridding approach of Tang~\emph{et al}~\cite{JustDiscretize},  two algorithms from the microscopy community (quickPALM and center of Gaussians), and also CGM and the gradient flow (GF) algorithm proposed by~\cite{Austrian}.    
The gridding approach approximately solves the continuous optimization problem~\eqref{cvxproblem} by discretizing the space $\Theta$ into a finite grid of candidate point source locations and running an $\ell_1$-regularized regression.  
In practice there is typically a small cluster of nonzero weights in the neighborhood of each true point source. With a fine grid, each of these clusters contains many nonzero weights, yielding many false positives.
 
To remove these false positives, Tang~\emph{et al} propose a heuristic post-processing step that involves taking the center of mass of each cluster.
This post-processing step is hard to understand theoretically, and does not perform well with a high-density of fluorophores.

\subsubsection{Implementation details} For this application, the minimization required in step \ref{gba} of \OurAlg~is not difficult: the parameter space is two-dimensional.  Coarse gridding followed by a local optimization method works well in theory and practice.  

For $\improve$ we use a standard constrained gradient method provided by the NLopt library~\cite{NLopt}.

\subsubsection{Evaluation}
We measure localization accuracy by computing the $F_1$ score, the harmonic mean of precision and recall, at varying radii.
Computing the precision and recall involves first matching estimated point sources to true point sources---a difficult task.
Fortunately, the SMLM challenge website~\cite{SMIChallenge} provides a stand-alone application that we use to compute the $F_1$ score.

We use a dataset of 12000 images that overlay to form simulated microtubules (see Figure~\ref{FigTubes}) available online at the SMLM challenge website~\cite{SMIChallenge}.
There are 81049 point sources in total, roughly evenly distributed across the images.  Figure~\ref{single} shows a typical image.
Each image covers an area 6400nm across, meaning each pixel is roughly 100nm by 100nm. 

Figure~\ref{FigLongSequence} compares the performance of \OurAlg, gridding, quickPALM, and center of Gaussians (CoG) on this dataset.  We match the performance of the gridding algorithm from~\cite{JustDiscretize}, and significantly beat both quickPALM and CoG.
Our algorithm analyses all images in well under an hour---significantly faster than the gridding approach of~\cite{JustDiscretize}.
Note that the gridding algorithm of~\cite{JustDiscretize} does not work without a post-processing step.

\begin{figure}
\centering
\begin{subfigure}{.4\textwidth}
\includegraphics[width=\textwidth]{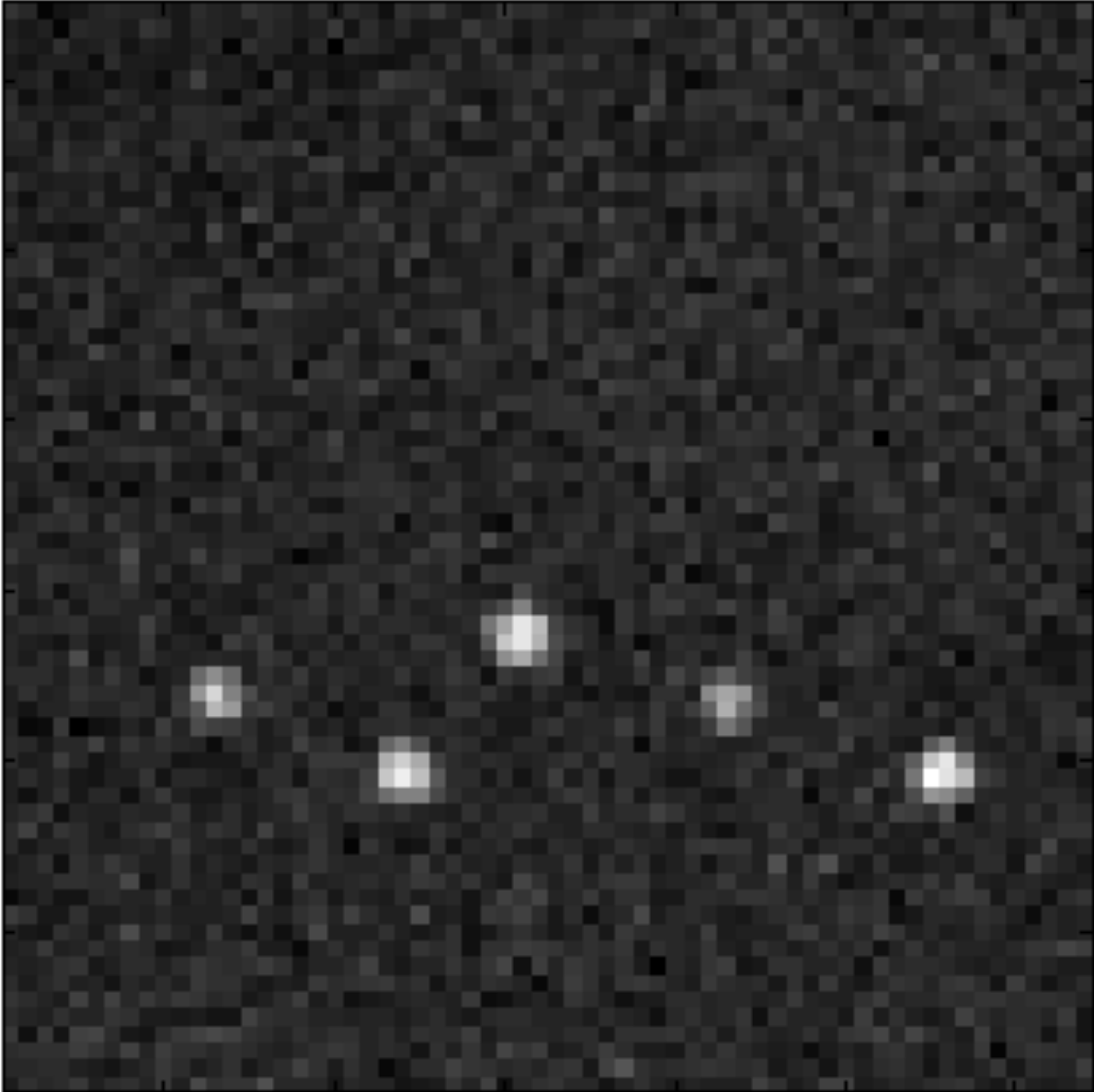}
\caption{}
\label{single}
\end{subfigure}
\begin{subfigure}{0.4\textwidth}
\includegraphics[width=\textwidth]{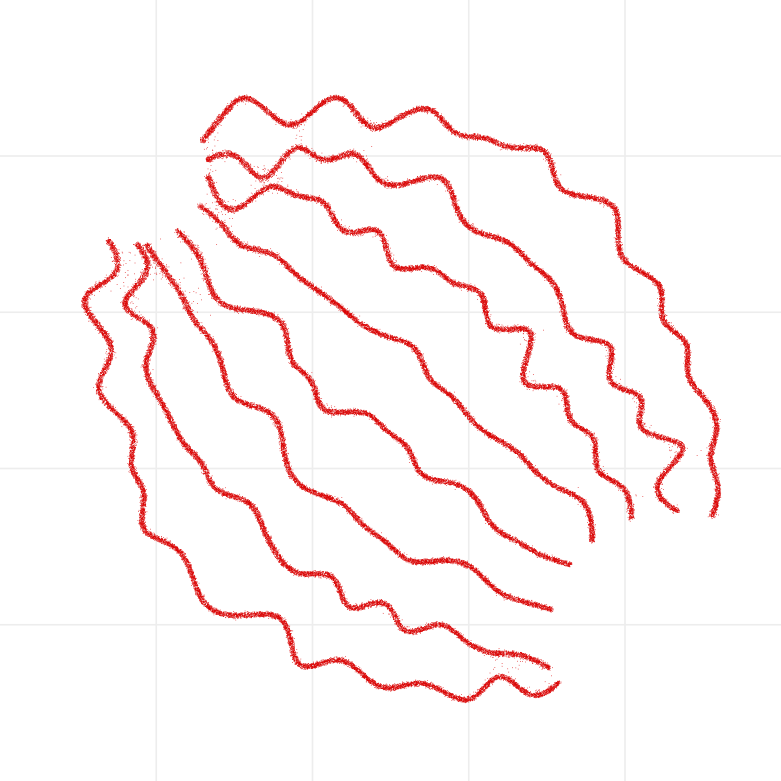}
\caption{}
\label{large}
\end{subfigure}
\caption{The long sequence dataset contains 12000 images similar to~(a).  The recovered locations for all the images are displayed in (b).}
\label{FigTubes}
\end{figure}

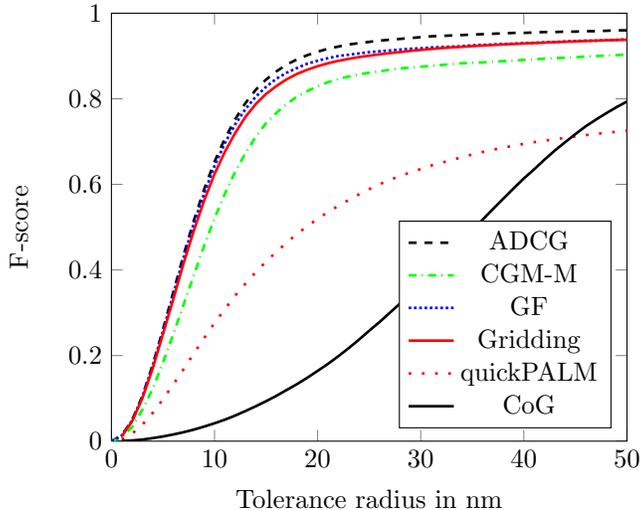
\begin{figure}
        \centering
\begin{tikzpicture}[scale=1] 
\begin{groupplot}[group style={group size=1 by 1,horizontal sep=1cm,xlabels at=edge bottom, ylabels at=edge left,xticklabels at=edge bottom},xlabel=Tolerance radius in nm,
        ylabel=F-score,legend pos = south east,ymin=-.1,ymax=1.1]
 \nextgroupplot[xmin=0,xmax=50, ymin=0,ymax=1]
 	\addplot [black,dashed,line width=1pt] table[col sep = comma, row sep = \\] {./csvFiles/smi_adcg.csv};\addlegendentry{\OurAlg}
	\addplot [green,dashdotted,line width=1pt] table[col sep = comma, row sep = \\] {./csvFiles/smi_fwam.csv};\addlegendentry{\FCFW}
 	\addplot [blue,densely dotted,line width=1pt] table[col sep = comma, row sep = \\] {./csvFiles/smi_gf.csv};\addlegendentry{GF}
	\addplot [red,line width=1pt] table[col sep = comma, row sep = \\] {./csvFiles/sparseSR.csv};\addlegendentry{Gridding}
        \addplot [red,loosely dotted,line width=1pt] table[col sep = comma, row sep = \\] {./csvFiles/quickPALM.csv};\addlegendentry{quickPALM}
        \addplot [black,line width=1pt] table[col sep = comma, row sep = \\] {./csvFiles/CoG.csv};\addlegendentry{CoG}
        \end{groupplot}
\end{tikzpicture}
\caption{{\bf Performance on bundled tubes: long sequence.}  F-scores at various radii for 6 algorithms. For reference, each image is 6400nm across, meaning each pixel has a width of 100nm. ADCG outperforms all competing methods on this dataset.}
\label{FigLongSequence}
\end{figure}

\subsection{Matrix completion}

As described in~\S\ref{SecExamples}, matrix completion is the task of estimating an approximately low rank matrix from some of its entries. We test our proposed algorithm on the Netflix Prize dataset, a standard benchmark for matrix completion algorithms. 

\subsubsection{Implementation details}

Although the parameter space for this example is high-dimensional we can still compute the steepest descent step over the space of measures.
The optimization problem we need to solve is:
 \[ \minimize_{\|a\|_2 = \|b\|_2 = 1} \langle \psi(a,b), \nu \rangle  = \langle M(ab^T), \nu \rangle = \langle ab^T, M^*(\nu) \rangle.\]

In other words, we need to find the unit norm, rank one matrix with highest inner product with the matrix $M^* \nu$. The solution to this problem is given by the top singular vectors of $M^* \nu$. Computing the top singular vectors using a Lanczos method is relatively easy as the matrix $M^* \nu$ is extremely sparse. 

Our implementation of $\improve$ takes a single step of gradient descent with line-search.

\subsubsection{Evaluation}
Our algorithm matches the state of the art for nuclear norm based approaches on the Netflix Prize dataset. Briefly, the task here is to predict the ratings 480,189 Netflix users give to a subset of 17,770 movies. One approach has been to phrase this as a matrix completion problem.
That is, to try to complete the 480,189 by 17,770 matrix of ratings from the observed entires. Following~\cite{jellyfish} we subtract the mean training score from all movies and truncate the predictions of our model to lie between $1$ and $5$.

Figure~\ref{matrixcompletion} shows root-mean-square error (RMSE) of our algorithm and other variants of the CGM on the Netflix probe set.
Again, \OurAlg~outperforms all other CGM variants.
Our algorithm takes over 7 hours to achieve the best RMSE---this could be improved with a more sophisticated implementation, or parallelization.

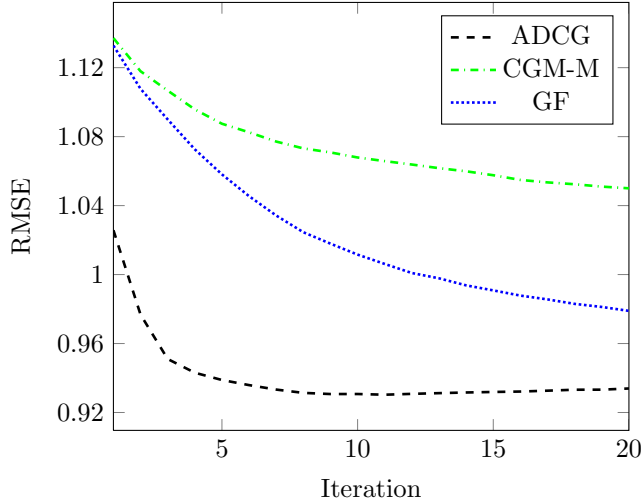
\begin{figure}
        \centering
\begin{tikzpicture}[scale=1] 
\begin{groupplot}[group style={group size=2 by 1,horizontal sep=1cm,xlabels at=edge bottom, ylabels at=edge left,xticklabels at=edge bottom},xlabel=Iteration,
        ylabel=RMSE,legend pos = north east,ytick={0.92,0.96,...,1.15}, xmin=1,xmax=20]
 \nextgroupplot[]
 	\addplot [black,dashed,line width=1pt] table[col sep = comma, row sep = \\,x index = {2}, y index = {1}] {./csvFiles/mc_adcg.csv};\addlegendentry{\OurAlg}	
 	\addplot [green,dashdotted,line width=1pt] table[col sep = comma, row sep = \\,x index = {2}, y index = {1}] {./csvFiles/mc_fwam.csv};\addlegendentry{\FCFW}	
  	\addplot [blue,densely dotted,line width=1pt] table[col sep = comma, row sep = \\,x index = {2}, y index = {1}] {./csvFiles/mc_gf.csv};\addlegendentry{GF}	

        \end{groupplot}
\end{tikzpicture}
\caption{{\bf RMSE on Netflix challenge dataset.}  ADCG significantly outperforms \FCFW.}
\label{matrixcompletion}
\end{figure}

\subsection{System identification}

In this section we apply our algorithms to identifying two single-input single-output systems from the DaISy collection~\cite{daisy}: the flexible robot arm dataset (ID 96.009) and the hairdryer dataset (ID 96.006).

\subsubsection{Implementation details}
While the parameter space is $6$-dimensional, which effectively precludes gridding, we can efficiently solve the minimization problem in step \eqref{lmo} of the ADCG. To do this, we grid only over $r$ and $\alpha$: the output is linear in the remaining parameters ($B$ and $x_0$) allowing us to analytically solve for the optimal $B$ and $x_0$ as a function of $r$, $\alpha$.

For $\improve$ we again use a standard box-constrained gradient method provided by the NLopt library~\cite{NLopt}.
\subsubsection{Evaluation}
Both datasets were generated by driving the system with a specific input and recording the output. 
The total number of samples is $1000$ in both cases.
Following~\cite{SysID} we identify the system using the first 300 time points and we evaluate performance by running the identified system forward for the remaining time points and compare our predictions to the ground truth.

\newcommand{\ypred}{y_{\text{pred}}}
\newcommand{\ymean}{y_{\text{mean}}}
We evaluate our predictions $\ypred$ using the score defined in~\cite{Lieven}.
The score is given by
\begin{equation}
\label{EqnScore}
\text{score} = 100\left(1 - \frac{\Vert \ypred  - y \Vert_2}{\Vert \ymean - y \Vert_2}\right),
\end{equation}
where $\ymean$ is the mean of the test set $y$.

Figure~\ref{RobotID} shows the score versus the number of sources as we run our algorithm.  For reference we display with horizontal lines the results of~\cite{Lieven}. \OurAlg~matches the performance of~\cite{Lieven} and exceeds that of all other CGM variants. Our simple implementation takes about an hour, which compares very poorly with the spectral methods in~\cite{Lieven} which complete in under a minute.

\begin{figure}
        \centering
\begin{tikzpicture}[scale=1] 
\begin{groupplot}[group style={group size=2 by 1,horizontal sep=1cm,xlabels at=edge bottom, ylabels at=edge left,xticklabels at=edge bottom},xlabel=Iteration,
        ylabel=score,legend pos = south west]
 \nextgroupplot[title={Hair dryer},ymin =0,ymax = 100,xmin=1, xmax=15]
 	\addplot [black,dashed,line width=1pt] table[col sep = comma, row sep = \\,x index = {0},y index = {1}] {./csvFiles/dryer_adfw.csv};\addlegendentry{\OurAlg}
	\addplot [green,dashdotted,line width=1pt] table[col sep = comma, row sep = \\,x index = {0},y index = {1}] {./csvFiles/dryer_fwam.csv};\addlegendentry{\FCFW}
 	\addplot [blue,densely dotted,line width=1pt] table[col sep = comma, row sep = \\,x index = {0},y index = {1}] {./csvFiles/dryer_gf.csv};\addlegendentry{GF}
	\addplot[mark = none,blue,domain=0:15] {86.3};
\nextgroupplot[title={Flexible robot arm},ymin =0,ymax = 100,xmin=1,xmax=15]
	\addplot [black,dashed,line width=1pt] table[col sep = comma, row sep = \\, x index = 0, y index = 1] {./csvFiles/robot_arm_adfw.csv};\addlegendentry{\OurAlg}
	\addplot [green,dashdotted,line width=1pt] table[col sep = comma, row sep = \\,x index = {0},y index = {1}] {./csvFiles/robot_arm_fwam.csv};\addlegendentry{\FCFW}
 	\addplot [blue,densely dotted,line width=1pt] table[col sep = comma, row sep = \\,x index = {0},y index = {1}] {./csvFiles/robot_arm_gf.csv};\addlegendentry{GF}
	\addplot[mark = none,blue,domain=-10:20] {96.0};
        \end{groupplot}
\end{tikzpicture}
\caption{{\bf Performance on DaISy datasets. } \OurAlg~outperforms other CGM variants and matches the nuclear-norm based technique of~\cite{SysID}.}
\label{RobotID}
\end{figure}
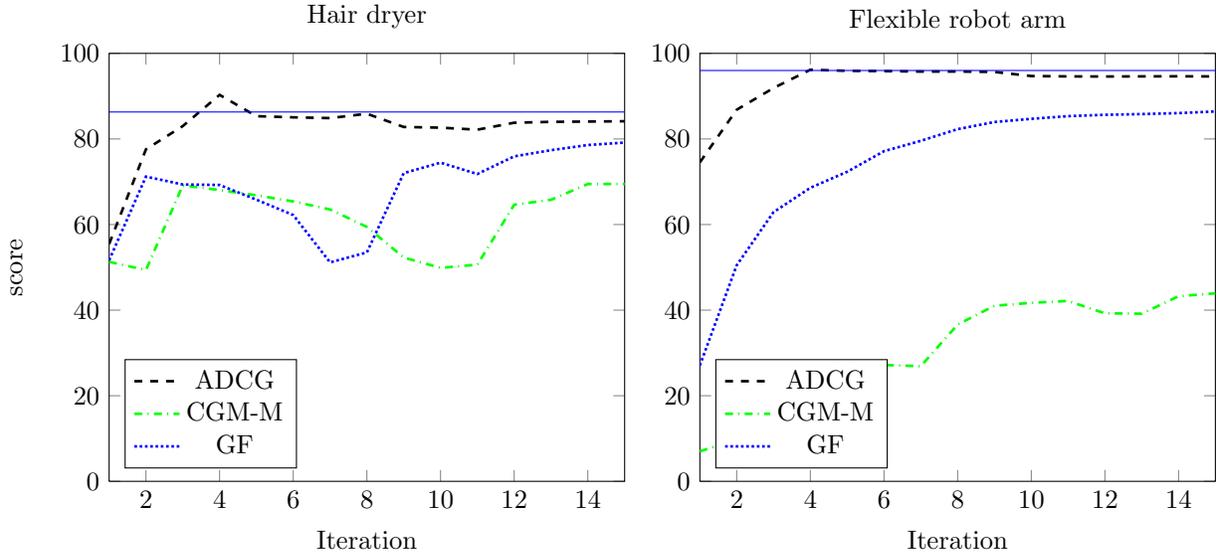

%%%%%%%%%%%%%%%%%%%%%%%%%%%%%%%%%%%%%%%%%%%

\section{Conclusions and future work}

As demonstrated in the numerical experiments of \S\ref{SecNumerics}, \OurAlg~achieves state of the art performance in superresolution fluorescence microscopy, matrix completion, and system identification, without the need for heuristic post-processing steps. The addition of the nonconvex local search step $\improve$ significantly improves performance relative to the standard conditional gradient algorithm in all of the applications investigated.  In some sense, we can understand~\OurAlg~as a method to rigorously control local search.  One could just start with a model expansion~\eqref{EqnForwardOperator} and perform nonconvex local search.  However, this fares far worse than~\OurAlg~in practice and has no theoretical guarantees. The \OurAlg~framework provides a clean way to generate a globally convergent algorithm that is practically efficient.  Understanding this coupling between local search heuristics and convex optimization leads our brief brief discussion of future work.

\paragraph{Tighten convergence analysis for \OurAlg.} The conditional gradient method is a robust technique, and adding our auxiliary local search step does not change its convergence rate.  However, in practice, the difference between the ordinary conditional gradient method, the fully corrective variants, and \OurAlg~are striking.
In many of our experiments, \OurAlg~outperforms the other variants by appreciable margins.  Yet, all of these algorithms share the same upper bound on their convergence rate.  A very interesting direction of future work would be to investigate if the bounds for \OurAlg~can be tightened at all to be more predictive of practical performance.  There may be connections between our algorithm and other alternating minimization techniques popular in matrix completion~\cite{keshavan2012efficient, jain2013low}, sparse coding~\cite{agarwal2013learning,arora2015simple}, and phase retrieval~\cite{netrapalli2013phase}, and perhaps the techniques from this area could be applied to our setting of sparse inverse problems.  

\paragraph{Relaxation to clustering algorithms.} Another possible connection that could be worth exploring is the connection between the CGM and clustering algorithms like k-means.  Theoretical bounds have been devised for initialization schemes for clustering algorithms that resemble the first step of CGM~\cite{Arthur07, Ostrovsky12}.  In these methods, k-means is initialized by randomly seeking the points that are farthest from the current centers.  This is akin to the first step of CGM which seeks the model parameters that best describe the residual error.  Once a good seeding is acquired, the standard Lloyd iteration for k-means can be shown to converge to the global optimal solution~\cite{Ostrovsky12}.   It is possible these analyses could be generalized to analyze our version of CGM or inspire new variants of the CGM.

\paragraph{Connections to cutting plane methods and semi-infinite programs.}
The standard Lagrangian dual of~\eqref{cvxproblem} is a semi-infinite program (SIP), namely an optimization problem with a finite dimensional decision variable but an infinite collection of constraints~\cite{hettich1993semi,shapiro2009semi}.  One of the most popular algorithmic techniques for SIP is the cutting plane method, and these methods qualitatively act very much like the CGM. Exploring this connection in detail could generate variants of cutting plane methods suited for continuous constraint spaces.  Such algorithms could be valuable tools for solving semi-infinite programs that arise in contexts disjoint from sparse inverse problems.

\paragraph{Other applications.}  
We believe that our techniques are broadly applicable to other sparse inverse problems, and hope that future work will explore the usefulness of~\OurAlg~in  areas unexplored in this paper.  To facilitate the application of~\OurAlg~to more problems, such as those described in \S\ref{SecExamples}, we have made our code publicly available on GitHub.  As described in \S\ref{SecAlgorithms}, implementing \OurAlg~for a new application essentially requires only two user-specified subroutines: one routine that evaluates the the measurement model and its derivatives at a specified set of weights and model parameters, and one that approximately solves the linear minimization in step~\ref{gba} of \OurAlg.  We aim to investigate several additional applications in the near future to test the breadth of the efficacy of \OurAlg.

 \section*{Acknowledgements}
We would like to thank Elina Robeva and Stephen Boyd for many useful conversations about this work.

BR is generously supported by ONR awards N00014-11-1-0723 and N00014-13-1-0129, NSF awards CCF-1148243 and CCF-1217058, AFOSR award FA9550-13-1-0138, and a Sloan Research Fellowship.  GS was generously supported by NSF award CCF-1148243.  NB was generously supported by a Google Fellowship from the Hertz Foundation.  This research is supported in part by NSF CISE Expeditions Award CCF-1139158, LBNL Award 7076018, and DARPA XData Award FA8750-12-2-0331, and gifts from Amazon Web Services, Google, SAP, The Thomas and Stacey Siebel Foundation, Adatao, Adobe, Apple, Inc., Blue Goji, Bosch, C3Energy, Cisco, Cray, Cloudera, EMC2, Ericsson, Facebook, Guavus, HP, Huawei, Informatica, Intel, Microsoft, NetApp, Pivotal, Samsung, Schlumberger, Splunk, Virdata and VMware.

\printbibliography

\end{document}